\documentclass{amsart}
\usepackage[english]{babel}
\usepackage[utf8]{inputenc}
\usepackage{amsmath}
\usepackage{amsthm}
\usepackage{comment}

\usepackage{amssymb}
\usepackage{tikz-cd}
\usepackage{enumitem}
\usepackage[backref=page]{hyperref}
\usepackage{cleveref}
\usepackage{mathtools}
\usepgflibrary{bbox}

\newtheorem{thm}{Theorem}[section]
\newtheorem{prop}[thm]{Proposition}
\newtheorem{lemma}[thm]{Lemma}

\theoremstyle{definition}

\theoremstyle{remark}
\newtheorem{rmk}[thm]{Remark}
\numberwithin{equation}{section}

\newcommand{\Z}{\mathbb Z}

\newcommand{\Br}{\operatorname{Br}}
\newcommand{\NS}{\operatorname{NS}}
\newcommand{\im}{\operatorname{im}}
\newcommand{\coker}{\operatorname{coker}}
\renewcommand{\phi}{\varphi}

\title{On the crystalline Artin--Mumford invariant}
\author{Margherita Pagano, Domenico Valloni}
\date{}

\begin{document}
\maketitle

\begin{abstract}
Let $k$ be an algebraically closed field of characteristic $p>0$ and let
$W=W(k)$. We prove that the torsion subgroup of $H^3_{\mathrm{cris}}(X/W)$
is a birational invariant of smooth projective varieties over $k$. The proof combines the action of correspondences on
Hodge--Witt cohomology and on the first Nygaard piece of the de Rham--Witt complex.
\end{abstract}

\section{Introduction}

Let $k$ be an algebraically closed field of characteristic $p>0$, and let
$W=W(k)$ be the ring of Witt vectors of $k$. For a smooth projective variety
$X/k$, we write $H^i_{\mathrm{cris}}(X/W)$ for its $i$-th crystalline cohomology group.

The goal of this article is to prove that
$\operatorname{Tors}H^3_{\mathrm{cris}}(X/W)$ is a birational invariant of
smooth projective varieties over $k$. This is the crystalline analogue of
the Artin--Mumford invariant
$\operatorname{Tors}H^3_{\mathrm{sing}}(X,\mathbf Z)$ for a smooth projective variety $X$ defined over the complex numbers. Artin and Mumford used this invariant in
\cite{ArtinMumford} to construct unirational threefolds, which are not
rational. One can prove its birational invariance using the identification
$\operatorname{Tors}H^3_{\mathrm{sing}}(X,\mathbf Z)
\simeq\operatorname{Br}(X)/\operatorname{Br}(X)_{\mathrm{div}}$
and the birational invariance of the Brauer group
\cite{GrothendieckBrauerIII}, or by using weak factorization
\cite{AKMW} and the blow-up formula for cohomology. These arguments do not
directly apply to crystalline cohomology, in fact: the description through the
Brauer group does not recover crystalline torsion, and
weak factorization is not available in positive characteristic. In \cite{AmbrosiValloni}, Ambrosi and the second author used crystalline
torsion to study reduction modulo $p$ of Noether's problem. They did not prove its birational invariance, but established a vanishing
theorem for $H^3_{\mathrm{cris}}(X/W)[p]$ under hypotheses satisfied by
stably rational varieties. Ambrosi and Ancona subsequently used
crystalline torsion in \cite{AmbrosiAncona} to prove an Artin--Mumford
criterion for conic bundles in characteristic two. These works provided
obstructions to stable rationality, while the birational invariance of
$\operatorname{Tors}H^3_{\mathrm{cris}}(X/W)$ was left open. We prove the
following.

\begin{thm}\label{main thm}
Let $X$ and $Y$ be smooth projective varieties over $k$. If $X$ and $Y$ are
birational, then
\[
\operatorname{Tors}H^3_{\mathrm{cris}}(X/W)
\simeq
\operatorname{Tors}H^3_{\mathrm{cris}}(Y/W).
\]
\end{thm}
\subsection{Strategy of the proof}
Let $f:X\dashrightarrow Y$ be a birational map. The universal coefficient sequence
for crystalline cohomology gives a commutative diagram with exact rows
\begin{equation}\label{diagram: universal coefficient}
\begin{tikzcd}
0 \arrow[r]
& H^2_{\mathrm{cris}}(X/W)/p \arrow[r] \arrow[d,"\alpha"]
& H^2_{\mathrm{dR}}(X/k) \arrow[r] \arrow[d,"\beta"]
& H^3_{\mathrm{cris}}(X/W)[p] \arrow[r] \arrow[d,"\gamma"]
& 0 \\
0 \arrow[r]
& H^2_{\mathrm{cris}}(Y/W)/p \arrow[r]
& H^2_{\mathrm{dR}}(Y/k) \arrow[r]
& H^3_{\mathrm{cris}}(Y/W)[p] \arrow[r]
& 0.
\end{tikzcd}
\end{equation}
Here $\alpha$, $\beta$, and $\gamma$ are the maps induced by the closure of the graph of $f$; the compatibility among these maps is discussed in Section \ref{sec: correspondences}. We are going to prove that $\gamma$
is injective, and then apply the same argument to the transpose; see
Section~\ref{sec: final proof}. The snake lemma gives an exact sequence
\[
\begin{tikzcd}[
  /pgf/bezier bounding box=true,
  column sep=2em,
  row sep=3em
]
0 \arrow[r]
  & \ker\alpha \arrow[r]
  & \ker\beta \arrow[r]
  & \ker\gamma
    \arrow[dlll, out=0, in=180, looseness=2,
           "\delta" description] \\
\operatorname{coker}\alpha \arrow[r]
  & \operatorname{coker}\beta \arrow[r]
  & \operatorname{coker}\gamma \arrow[r]
  & 0.
\end{tikzcd}
\]

Thus, the injectivity of $\gamma$ follows from the following two facts
\begin{equation}\label{eq: two comparisons introduction}
\ker(\alpha)\longrightarrow\ker(\beta)\text{ is surjective},
\qquad
\operatorname{coker}(\alpha)\longrightarrow\operatorname{coker}(\beta)\text{ is injective}.
\end{equation}
We shall prove both statements. Our proof uses in a crucial way the action of correspondences on
Hodge--Witt cohomology constructed by Chatzistamatiou and R\"ulling
\cite{CR11,CR12}, together with the birational invariance of the
structure-sheaf cohomology proved by the same authors.

Concerning de Rham cohomology, the birational invariance of the cohomology of the structure sheaf with some $p$-linear algebra arguments allow us to prove the following.

\begin{prop}\label{prop: logarithmic description de Rham}
Let $f\colon X \dashrightarrow Y$ be a birational morphism between smooth projective variety over an algebraically closed field $k$. Then, there is an isomorphism
\[
\ker(\beta)\simeq
\ker\bigl(H^1(X,\Omega_{X,\log}^1)\to
H^1(Y,\Omega_{Y,\log}^1)\bigr)\otimes_{\mathbf F_p}k,
\]
and an injection
\[
\operatorname{coker}\bigl(H^1(X,\Omega_{X,\log}^1)\to
H^1(Y,\Omega_{Y,\log}^1)\bigr)\otimes_{\mathbf F_p}k
\hookrightarrow\operatorname{coker}(\beta).
\]
\end{prop} 
The Kummer sequence and the birational invariance
of the Brauer group then show that the kernel is generated by divisor
classes. Since these classes lift to crystalline cohomology, they give
the required surjection at the level of the kernels. 

Concerning crystalline cohomology, we use the first
piece of the Nygaard filtration \cite{BMS2} to obtain a
comparison between $\mathcal{C}_{\mathrm{cris}} \coloneqq \coker(H^2_{\mathrm{cris}}(X/W)\rightarrow H^2_{\mathrm{cris}}(Y/W))$ and weight-one flat (i.e. syntomic) cohomology. 
\begin{prop}\label{prop: intro cokernel comparison}
Let $f\colon X \dashrightarrow Y$ be a birational morphism between smooth projective varieties over an algebraically closed field $k$. There is an isomorphism
\begin{equation}\label{eq: integral crystalline cokernel syntomic}
\coker\left(H^2(X,\Z_p(1))\rightarrow H^2(Y,\Z_p(1))\right)\otimes_{\mathbf Z_p}W
\xrightarrow{\simeq}   \mathcal{C}_{\mathrm{cris}}.
\end{equation}
\end{prop}
Reducing modulo $p$ gives an isomorphism \[
\coker(\alpha)\simeq \operatorname{coker}\bigl(H^1(X,\Omega_{X,\log}^1)\to
H^1(Y,\Omega_{Y,\log}^1)\bigr)\otimes_{\mathbf F_p}k
\]
which combined with Proposition \ref{prop: logarithmic description de Rham} gives the desired injectivity. 

Applying these comparisons to the graph of a birational map and its
transpose proves injectivity on $H^3_{\mathrm{cris}}(-)[p]$ in both
directions. The same maps are then injective on the full torsion
subgroups, since any non-zero kernel would contain an element killed
by $p$. The theorem follows because these torsion subgroups have finite
length over $W$.

\subsection*{Previous works.} In \cite{AmbrosiValloni}, the second named author together with Ambrosi used $\operatorname{Tors}H^3_{\mathrm{cris}}(-/W)$ to show that for a smooth and proper scheme $\mathcal{X}$ over a complete discrete valuation ring $R$ of mixed characteristic $(0,p)$, the rationality of the special fibre gives torsion freeness of the third étale cohomology group with $\Z_p$ coefficients on the generic fibre. The core of their work consists in showing that smooth and proper rational varieties in positive characteristic have torsion-free third crystalline cohomology group. To conclude, they use \cite{BMS1}, to compare the torsion in the étale cohomology of the generic fibre with the torsion in the crystalline cohomology of the special fibre. 

In \cite[Section 2]{BMS1}, the authors give an example of a smooth and proper surface $\mathcal{X}$ over a complete discrete valuation ring $R$ of mixed characteristic $(0,p)$, with residue field $k$ and fraction field $K$, for which $H^2_{\text{ét}}(\mathcal{X}_K,\Z_p)$ is torsion-free, while $H^2_{\mathrm{cris}}(\mathcal{X}_k,W)$ has nontrivial torsion. Poincaré duality implies the same statement at the level of the third (étale and crystalline) cohomology group, showing that in the mixed characteristic setting, the third crystalline torsion might be strictly larger than the torsion in the third étale cohomology of the generic fibre. 

Finally, in \cite{AmbrosiAncona} Ambrosi and Ancona use the torsion in the third crystalline cohomology group to give an Artin-Mumford criterion for conic bundles in characteristic $2$, by building non-zero $2$-torsion elements in the third crystalline cohomology group. Moreover, in analogy with the Artin--Mumford invariant, they show that also the presence of torsion in $H^3_{\mathrm{cris}}(-)$ gives an obstruction to the decomposition of the diagonal, see \cite[Proposition~7.6]{AmbrosiAncona}.
\subsection*{Outline of the paper} In Section \ref{sec: correspondences}, we give a brief overview on the action of correspondences on de Rham--Witt and crystalline cohomology, discussing compatibility with all the maps that we will need to prove our main result. In Section \ref{sec: kernel}, we give a proof of Proposition \ref{prop: logarithmic description de Rham}, while in Section \ref{sec: nygaard} we deal with the proof of Proposition \ref{prop: intro cokernel comparison}. Finally, in Section \ref{sec: final proof} we put everything together to prove Theorem \ref{main thm}. We conclude the paper with an appendix in which we collect some technical results used in the text.

\subsection*{Acknowledgment}
The first named author was founded by Imperial College London and ICERM (Brown University) during various stages of this project. Part of this material is based upon work supported by the National Science Foundation under Grant No. DMS-2424556 while the author was in residence at the Institute for Computational and Experimental Research in Mathematics in Providence, RI, during the Computations on K3 Surfaces and Related Varieties program.

AI (ChatGPT Sol) checked for typos, and helped with editing and references. It also spotted a small imprecision in the proof of Lemma~\ref{lemma: divided Frobenius low degrees} which was then easily corrected. 
\section{Correspondences and cohomology}\label{sec: correspondences}

We collect the correspondence formalism used in the proof. Our conventions
follow Chatzistamatiou--R\"ulling \cite{CR11,CR12}.

\subsection{Cycle classes and the correspondence action}

Let $X$ and $Y$ be smooth projective $k$-varieties, with $X$ of pure
dimension $d$, and let
\[
\Gamma\in \operatorname{CH}^{d+r}(X\times Y).
\]
For a Weil cohomology theory $H^*(-)$, the correspondence $\Gamma$ gives an action, the so-called \emph{action of correspondence}, on cohomology. More precisely, writing $p_1,p_2$ for the two projections, the usual correspondence formula is
\begin{equation}\label{eq: correspondence formula}
\mathrm{H}^i(X)\rightarrow \mathrm{H}^{i+2r}(Y), \quad \xi\mapsto\Gamma_*(\xi):=p_{2*}\bigl(p_1^*\xi\cup\operatorname{cl}(\Gamma)\bigr).
\end{equation}
Even if crystalline cohomology is not a Weil cohomology theory, because it does not have coefficients in a field,  the same procedure gives an action of correspondences on it. In fact, the crystalline cycle classes and the corresponding Gysin maps and proper
pushforwards were constructed by Berthelot in the smooth case and then by Gros in general; see
\cite{Berthelot,Gros}. 

Chatzistamatiou and R\"ulling construct the action of proper correspondences
on relative Hodge--Witt cohomology in \cite[Section~3]{CR12}, using the notion of weak cohomology theory with support, introduced in \cite{CR11}. At finite Witt
level, the derived correspondence morphism is described in
\cite[Sections~5.1.5--5.1.9]{CR12}: it is obtained from pullback, cup product
with the cycle class, and proper pushforward, exactly as in
\eqref{eq: correspondence formula}. In particular, a correspondence of degree
$r$ induces maps
\begin{equation}\label{eq: Witt vectors correspondence}
    H^q(X,W_n\Omega_X^j)\longrightarrow
    H^{q+r}(Y,W_n\Omega_Y^{j+r}).
\end{equation}
For $n=1$, this agrees with the action on Hodge cohomology constructed in
\cite{CR11}; see also \cite[Lemma~5.1.9]{CR12}.

From now on, we use degree-zero correspondences.  
In \cite{CR12}, they show that the action of correspondence is compatible with $F,V,d$. Moreover, it recover the usual crystalline correspondence action via the isomorphism
\[
\mathbb H^m(X,W\Omega_X^\bullet)\simeq \mathrm{H}^m_{\mathrm{cris}}(X/W)
\]
proven in \cite[II, Theorem~1.4]{Illusie}. 
In \cite{Gros}, the author proves the compatibility of cycle classes and
Gysin maps, see \cite[II, Sections~3.4 and~4.1]{Gros}. The compatibility
of traces and duality can be found in \cite[I, Section~5; III, Proposition~5.6]{Ekedahl}.

\subsection{Compatibility with closed, exact and logarithmic forms}
To avoid confusion, we will use $T$ to denote any smooth and proper variety over $k$. We define
\[
Z_T^i:=\ker(d:\Omega_T^i\rightarrow\Omega_T^{i+1}),
\qquad
B_T^i:=\operatorname{im}(d:\Omega_T^{i-1}\rightarrow\Omega_T^i).
\]
The ordinary Cartier operator gives an exact sequence
\begin{equation}\label{eq: cartier operator}
0\longrightarrow B_T^1\longrightarrow Z_T^1
\xrightarrow{C}\Omega_T^1\longrightarrow0.
\end{equation}
The map $C$ is $\sigma^{-1}$-semilinear, where
$\sigma(a)=a^p$. We also have, on the \'etale site, the logarithmic Cartier
sequence
\begin{equation}\label{eq: logarithmic forms}
0\longrightarrow\Omega_{T,\log}^1\longrightarrow Z_T^1
\xrightarrow{\iota-C}\Omega_T^1\longrightarrow0,
\end{equation}
where $\iota$ is the natural inclusion. We refer to
\cite{Illusie,IllusieRaynaud}; see also
\cite[Theorem~1.1.3 and Proposition~1.1.7(iii)]{Zhao} for these two
sequences. Cohomology of logarithmic forms is taken on the \'etale site;
for the other sheaves we use the equivalent Zariski cohomology.

The compatibility of the action of the correspondence on the de Rham--Witt complex with $F,V$ and $d$ induces compatible maps on the long exact sequence in cohomology arising from both the Cartier short exact sequences \eqref{eq: cartier operator} and the logarithmic Cartier sequence \eqref{eq: logarithmic forms}. In particular, the action of correspondence is compatible with the differentials appearing in the conjugate spectral sequence, see Section \ref{sec: kernel}.

\subsection{Compatibility with the Nygaard filtration}

The first Nygaard subcomplex is
\[
\mathcal N^{\geq1}W\Omega_T^\bullet
=
\bigl[VW\mathcal O_T\longrightarrow W\Omega_T^1
\longrightarrow W\Omega_T^2\longrightarrow\cdots\bigr], 
\]
see \cite[Definition~8.1]{BMS2}. We use the
correspondence maps on this complex before taking hypercohomology.
For clarity, after identifying $VW\mathcal O_T$ with $W\mathcal O_T$
through $V$, its additive complex is
$[W\mathcal O_T\xrightarrow{dV}W\Omega_T^1\xrightarrow{d}\cdots]$,
with the corresponding Frobenius twist on the degree-zero scalar
structure. The inclusion is $V$ in degree zero and the identity in
positive degrees; divided Frobenius is the identity in degree zero
and $p^{j-1}F$ in degree $j\geq1$. Again, the compatibility of the action of the correspondences with $F,V,d$ at the level of derived categories, implies that such action respects 
$\mathcal N^{\geq i}W\Omega_T^\bullet$ as well. The identifications of the Nygaard graded pieces with the conjugate
truncations in \cite[Lemma~8.2]{BMS2}, and the Hodge quotients in
\cite[Lemma~8.3]{BMS2}, give compatibility with the conjugate spectral
sequence.

\section{Correspondences on de Rham cohomology}\label{sec: kernel}

 Our main goal in this section is to prove Proposition \ref{prop: logarithmic description de Rham}.
We first compare the de Rham kernel and cokernel with their Hodge
counterparts. Put
\begin{align*}
K&:=\ker\bigl(H^1(X,\Omega_X^1)\longrightarrow H^1(Y,\Omega_Y^1)\bigr),\\
\mathcal C&:=\operatorname{coker}\bigl(
H^1(X,\Omega_X^1)\longrightarrow H^1(Y,\Omega_Y^1)\bigr).
\end{align*}

\begin{lemma}\label{lemma: de Rham kernel cokernel H11}
There are natural $\sigma^{-1}$-semilinear isomorphisms
\[
\ker(\beta)\xrightarrow{\simeq}K,
\qquad
\operatorname{coker}(\beta)\xrightarrow{\simeq}\mathcal C.
\]
\end{lemma}

\begin{proof}
Consider the conjugate spectral sequences
\[
{}'E_2^{a,b}(T)=H^a\bigl(T,\mathcal H^b(\Omega_T^\bullet)\bigr)
\Longrightarrow H^{a+b}_{\mathrm{dR}}(T/k),
\qquad T=X,Y.
\]
By Cartier, ${}'E_2^{a,b}(T)$ identifies $\sigma^{-1}$-semilinearly with
$H^a(T,\Omega_T^b)$. Put
\[
A_T:=\ker\bigl(H^2_{\mathrm{dR}}(T/k)\to{}'E_2^{0,2}(T)\bigr).
\]
The seven-term exact sequence gives
\[
\begin{tikzcd}
{{}'E_2^{2,0}(X)} \arrow[r] \arrow[d,"\cong"]
& A_X \arrow[r] \arrow[d]
& {{}'E_2^{1,1}(X)} \arrow[r] \arrow[d]
& {{}'E_2^{3,0}(X)} \arrow[d,"\cong"] \\
{{}'E_2^{2,0}(Y)} \arrow[r]
& A_Y \arrow[r]
& {{}'E_2^{1,1}(Y)} \arrow[r]
& {{}'E_2^{3,0}(Y)}.
\end{tikzcd}
\]
The outer vertical arrows identify, through Cartier, with the maps on
$H^2(\mathcal O)$ and $H^3(\mathcal O)$. The transpose gives a commuting
reverse diagram, with inverse maps on these outer terms. A diagram chasing argument, see
Lemma~\ref{lemma: identification of kernel and cokernels},
identifies the kernels and cokernels of the two middle vertical arrows.

We are left to show that there is an isomorphism 
\[
\ker\left((H^2_{\mathrm{dR}}(X/k)\rightarrow H^2_{\mathrm{dR}}(Y/k)\right)\simeq \ker\left(A_X\rightarrow A_Y \right)
\]
and 
\[
\coker\left((H^2_{\mathrm{dR}}(X/k)\rightarrow H^2_{\mathrm{dR}}(Y/k)\right)\simeq \coker\left(A_X\rightarrow A_Y \right).
\]
In fact, the result then follows from the isomorphism ${}'E_2^{1,1}(T) \cong H^1(T^{(p)}, \Omega^1_{T^{(p)}})$ given by Cartier. To prove this, consider the image 
\[
I_T:=\operatorname{im}\bigl(H^2_{\mathrm{dR}}(T/k)\to{}'E_2^{0,2}(T)\bigr);
\]
since global two-forms are birational invariants, and the graph and transpose
preserve $I_T$, we get that $I_X\to I_Y$ is an isomorphism. Applying the snake
lemma to
\[
0\to A_T\to H^2_{\mathrm{dR}}(T/k)\to I_T\to0
\]
gives the result.
\end{proof}

\begin{lemma}\label{lemma: B1 birational invariance}
For every $i\geq0$, the correspondence induces an isomorphism
\[
H^i(X,B_X^1)\xrightarrow{\simeq}H^i(Y,B_Y^1),
\]
and the same holds for the transpose correspondence.
\end{lemma}

\begin{proof}
For $T=X,Y$ there is an exact sequence of additive sheaves
\[
0\longrightarrow\mathcal O_T\xrightarrow{a\mapsto a^p}\mathcal O_T
\xrightarrow{d}B_T^1\longrightarrow0.
\]
The correspondence action on the cohomology of the structure sheaf, and on its
Frobenius twist, are isomorphisms which commute with Frobenius. The result then
follows from the associated long exact sequences together with the birational invariance of $H^i(T, \mathcal{O}_T)$.
\end{proof}

We also use that
\[
H^0(X,Z_X^i)\simeq H^0(Y,Z_Y^i),
\qquad
H^0(X,B_X^i)\simeq H^0(Y,B_Y^i).
\]
Indeed, the first assertion follows from birational invariance of global
forms and compatibility with $d$, and the second from
$B_T^i=\ker(C:Z_T^i\to\Omega_T^i)$.

Define
\begin{align*}
K'&:=\ker\bigl(H^1(X,Z_X^1)\longrightarrow H^1(Y,Z_Y^1)\bigr),\\
\mathcal C'&:=\operatorname{coker}\bigl(
H^1(X,Z_X^1)\longrightarrow H^1(Y,Z_Y^1)\bigr).
\end{align*}
Applying \eqref{eq: cartier operator} gives
\[
\begin{tikzcd}
{H^1(X,B_X^1)} \arrow[r] \arrow[d,"\cong"]
& {H^1(X,Z_X^1)} \arrow[r,"C"] \arrow[d]
& {H^1(X,\Omega_X^1)} \arrow[r] \arrow[d]
& {H^2(X,B_X^1)} \arrow[d,"\cong"] \\
{H^1(Y,B_Y^1)} \arrow[r]
& {H^1(Y,Z_Y^1)} \arrow[r,"C"]
& {H^1(Y,\Omega_Y^1)} \arrow[r]
& {H^2(Y,B_Y^1)}.
\end{tikzcd}
\]
The transpose again gives a commuting reverse diagram, with inverse
outer maps. Hence Lemma~\ref{lemma: identification of kernel and cokernels}
gives $\sigma^{-1}$-semilinear isomorphisms
\begin{equation}\label{eq: C on K and coker}
C:K'\xrightarrow{\simeq}K,
\qquad
C:\mathcal C'\xrightarrow{\simeq}\mathcal C.
\end{equation}

The exact sequence
\[
0\longrightarrow Z_T^1\xrightarrow{\iota}\Omega_T^1
\xrightarrow{d}B_T^2\longrightarrow0
\]
gives maps
\[
\iota:K'\longrightarrow K,
\qquad
\iota:\mathcal C'\longrightarrow\mathcal C.
\]
The first is injective. Indeed, in the corresponding long exact sequences the
maps on $H^0(\Omega^1)$ and on $H^0(B^2)$ are isomorphisms, so the images of
the boundary maps $H^0(B^2)\to H^1(Z^1)$ are identified. Since $K'$ and $K$
have the same dimension by \eqref{eq: C on K and coker}, we obtain
\begin{equation}\label{eq: iota K iso}
\iota:K'\xrightarrow{\simeq}K.
\end{equation}

In particular, $\iota\circ C^{-1}:K\to K$ is a $\sigma$-semilinear automorphism. Hence, Lemma~\ref{lemma: semilinear algebra finite length} with
$n=1$ gives
\begin{equation}\label{eq: fixed points K}
\ker(\iota\circ C^{-1}-\mathrm{id}:K\to K)
\otimes_{\mathbf F_p}k\simeq K,
\end{equation}
an injective map 
\begin{equation}\label{eq: fixed points C}
\ker(\iota\circ C^{-1}-\mathrm{id}:\mathcal C\to\mathcal C)
\otimes_{\mathbf F_p}k\hookrightarrow\mathcal C
\end{equation}
and a surjective map 
\begin{equation}\label{eq: surjecivity K' to K}
    \iota\circ C^{-1}-\mathrm{id}:K\twoheadrightarrow K.
\end{equation}

\begin{rmk}
We do not need, and the preceding diagram does not imply, that
$\iota:\mathcal C'\to\mathcal C$ is an isomorphism.
\end{rmk}

We now use the logarithmic Cartier sequence \eqref{eq: logarithmic forms}.
It gives the diagram
\begin{equation}\label{diagram: logarithmic forms}
\begin{tikzcd}
{H^0(X,\Omega_X^1)} \arrow[r] \arrow[d,"\cong"]
& {H^1(X,\Omega_{X,\log}^1)} \arrow[r] \arrow[d]
& {H^1(X,Z_X^1)} \arrow[r,"\iota-C"] \arrow[d]
& {H^1(X,\Omega_X^1)} \arrow[d] \\
{H^0(Y,\Omega_Y^1)} \arrow[r]
& {H^1(Y,\Omega_{Y,\log}^1)} \arrow[r]
& {H^1(Y,Z_Y^1)} \arrow[r,"\iota-C"]
& {H^1(Y,\Omega_Y^1)}.
\end{tikzcd}
\end{equation}

\begin{lemma}\label{lemma: logarithmic kernel and cokernel}
The diagram \eqref{diagram: logarithmic forms} induces an isomorphism
\[
\ker\bigl(H^1(X,\Omega_{X,\log}^1)\to
H^1(Y,\Omega_{Y,\log}^1)\bigr)
\simeq\ker(\iota\circ C^{-1}-\mathrm{id}:K\to K).
\]
The natural map between cokernels
\[
\operatorname{coker}\bigl(H^1(X,\Omega_{X,\log}^1)\to
H^1(Y,\Omega_{Y,\log}^1)\bigr)
\longrightarrow\mathcal C'
\]
is injective, with image contained in
$\ker(\iota-C:\mathcal C'\to\mathcal C)$.
\end{lemma}

\begin{proof}
The short exact sequence \eqref{eq: logarithmic forms} induces, for every smooth $k$-scheme $T$ a map 
\[
j_T:H^1(T,\Omega_{T,\log}^1)\to H^1(T,Z_T^1)
\]
such that $\ker(j_T)=\im(H^0(T,\Omega^1_T)\rightarrow H^1(T,\Omega^1_{T,\log}))$. Since the global sections of the sheaves $Z^1$ and $\Omega^1$ are
birational invariants, the induced map $\ker(j_X)\to\ker(j_Y)$ is an
isomorphism, see Lemma~\ref{lemma: identification of kernel and cokernels} applied to the diagram with horizontal arrows $H^0(T,Z^1_T)\rightarrow H^0(T,\Omega^1_T)\rightarrow \ker(j_T)\rightarrow 0$ with $T$ equal to $X$ and $Y$ respectively. 

The snake lemma applied to
\[
0\rightarrow \ker(j_T)\rightarrow H^1(T,\Omega_{T,\log}^1)
\rightarrow\operatorname{im}(j_T)\rightarrow0
\]
therefore identifies both logarithmic kernels and cokernels with
those of the induced map $\operatorname{im}(j_X)\to\operatorname{im}(j_Y)$. Now apply Lemma~\ref{lemma: snake kernel cokernel} to the exact rows
\[
0\longrightarrow\operatorname{im}(j_T)
\longrightarrow H^1(T,Z_T^1)
\xrightarrow{\iota-C}H^1(T,\Omega_T^1)
\longrightarrow H^2(T,\Omega_{T,\log}^1).
\]
Its surjectivity hypothesis holds because
\[
\iota-C=(\iota\circ C^{-1}-\mathrm{id})\circ C:
K'\longrightarrow K
\]
is surjective, cf. equation \eqref{eq: surjecivity K' to K}. Lemma~\ref{lemma: snake kernel cokernel} gives therefore the identification at the level of the kernels and the
injection of $\coker(\im(j_X)\rightarrow \im(j_Y))$ into $\ker(\iota-C:\mathcal C'\to\mathcal C)$.

Finally, the isomorphism $C:K'\xrightarrow{\simeq}K$ identifies
$\ker(\iota-C:K'\to K)$ with $\ker(\iota\circ C^{-1}-\mathrm{id}\colon K \rightarrow K)$, which completes the proof.

\end{proof}
\begin{proof}[Proof of Proposition~\ref{prop: logarithmic description de Rham}]
Lemma~\ref{lemma: logarithmic kernel and cokernel}
identifies the logarithmic kernel with the fixed subgroup in
\eqref{eq: fixed points K} and embeds the logarithmic cokernel into
the fixed subgroup in \eqref{eq: fixed points C}. Extending scalars
from $\mathbf F_p$ to $k$ therefore gives an isomorphism onto $K$
and an injection into $\mathcal C$, see Lemma \ref{lemma: semilinear algebra finite length}.

\end{proof}

\subsection{Surjectivity at the level of kernels}
We are ready to prove the surjectivity of the natural map $\ker(\alpha)\to\ker(\beta)$ in
\eqref{diagram: universal coefficient}.
For a smooth $k$-scheme $T$, the logarithmic differential identifies
\[
H^1(T,\Omega^1_{T,\log})\simeq H^2_{\mathrm{fppf}}(T,\mu_p);
\]
see \cite{Gros}. The Kummer sequence (in the fppf topology) therefore gives an exact sequence
\begin{equation}\label{eq: Kummer logarithmic}
0\longrightarrow \NS(T)/p
\longrightarrow H^1(T,\Omega^1_{T,\log})
\longrightarrow \Br(T)[p]\longrightarrow0,
\end{equation}
 Here we use that multiplication by $p$ is surjective on
$\operatorname{Pic}^0(T)(k)$, so that $\operatorname{Pic}(T)/p=\NS(T)/p$.

Since the Brauer group is a birational invariant, the graph correspondence
induces an isomorphism on the right-hand terms of \eqref{eq: Kummer logarithmic}.
Consequently
\begin{equation}\label{eq: log kernel NS}
\ker\bigl(H^1(X,\Omega^1_{X,\log})\to
H^1(Y,\Omega^1_{Y,\log})\bigr)
\simeq
\ker\bigl(\NS(X)/p\to\NS(Y)/p\bigr).
\end{equation}
Combining this with Proposition~\ref{prop: logarithmic description de Rham}
gives
\begin{equation}\label{eq: kernel beta NS}
\ker(\beta)\simeq
\ker\bigl(\NS(X)/p\to\NS(Y)/p\bigr)\otimes_{\mathbf F_p}k.
\end{equation}

\begin{prop}\label{prop: kernel comparison}
The natural map $\ker(\alpha)\to\ker(\beta)$ in
\eqref{diagram: universal coefficient} is surjective.
\end{prop}

\begin{proof}
The crystalline first Chern class gives a correspondence-compatible map
\[
c_1^{\mathrm{cris}}:\NS(X)\longrightarrow H^2_{\mathrm{cris}}(X/W),
\]
whose reduction modulo $p$ maps to the de Rham first Chern class under
the left-hand arrow in \eqref{diagram: universal coefficient};
see \cite{Berthelot,Gros}. If a class in $\NS(X)/p$ belongs to the kernel of
$\NS(X)/p\to\NS(Y)/p$, its crystalline cycle class modulo $p$ belongs to
$\ker(\alpha)$. By \eqref{eq: kernel beta NS}, these cycle classes span $\ker(\beta)$ over
$k$. Hence $\ker(\alpha)\to\ker(\beta)$ is surjective.
\end{proof}

\section{The Nygaard filtration on Crystalline cohomology}\label{sec: nygaard}
Let $T$ be any $k$-scheme. We denote by $H^i(T,\mathbf Z_p(1))$ the syntomic cohomology of
$R\varprojlim_n R\Gamma_{\mathrm{fppf}}(T,\mu_{p^n})$. Let $f\colon X\dashrightarrow Y$ be a birational morphism between smooth proper $k$-variety. We define
\begin{align*}
\mathcal C_{\mathrm{cris}}&:=\operatorname{coker}\bigl(
H^2_{\mathrm{cris}}(X/W)\to H^2_{\mathrm{cris}}(Y/W)\bigr),\\
\mathcal C_{\mathrm{syn}}&:=\operatorname{coker}\bigl(
H^2(X,\mathbf Z_p(1))\to H^2(Y,\mathbf Z_p(1))\bigr).
\end{align*}

The aim of this section is to prove the integral comparison of Proposition \ref{prop: intro cokernel comparison}: 
\[
\mathcal C_{\mathrm{syn}}\otimes_{\mathbf Z_p}W\simeq \mathcal C_{\mathrm{cris}}.
\]

\subsection{The two Nygaard maps}
For $T=X,Y$, let
$\mathcal N_T:=\mathcal N^{\geq1}W\Omega_T^\bullet$.
The inclusion and divided Frobenius give maps
\[
\iota,\varphi_{(1)}:\mathcal N_T\longrightarrow W\Omega_T^\bullet;
\]
see \cite[Definition~8.1]{BMS2}. We define
\begin{align*}
K_{\mathcal N}&:=\ker\bigl(
\mathbb H^2(X,\mathcal N_X)\to\mathbb H^2(Y,\mathcal N_Y)\bigr),\\
\mathcal C_{\mathcal N}&:=\operatorname{coker}\bigl(
\mathbb H^2(X,\mathcal N_X)\to\mathbb H^2(Y,\mathcal N_Y)\bigr),\\
K_{\mathrm{cris}}&:=\ker\bigl(
H^2_{\mathrm{cris}}(X/W)\to H^2_{\mathrm{cris}}(Y/W)\bigr)
\end{align*}
where all the maps are given the the correspondence $\Gamma_f$, as explained in Section \ref{sec: correspondences}.
\begin{prop}\label{prop: nygaard comparison maps}
The maps $\iota$ and $\varphi_{(1)}$ induce isomorphisms $K_{\mathcal N} \xrightarrow{\sim} K_{\mathrm{cris}}$ and $\mathcal C_{\mathcal N} \xrightarrow{\sim} \mathcal C_{\mathrm{cris}}$.
\end{prop}
\begin{proof}[Proof for $\iota$]
We begin the proof with the map $\iota$. Since $W\Omega_T^\bullet/\mathcal N_T=\mathcal O_T[0]$, the long exact
cohomology sequences give
\begin{equation}\label{eq: nygaard iota les}
\begin{tikzcd}[column sep=small]
{H^1(X,\mathcal O_X)} &
{\mathbb H^2(X,\mathcal N_X)} &
{H^2_{\mathrm{cris}}(X/W)} &
{H^2(X,\mathcal O_X)} \\
{H^1(Y,\mathcal O_Y)} &
{\mathbb H^2(Y,\mathcal N_Y)} &
{H^2_{\mathrm{cris}}(Y/W)} &
{H^2(Y,\mathcal O_Y)}
\arrow[from=1-1,to=1-2]
\arrow["\iota",from=1-2,to=1-3]
\arrow[from=1-3,to=1-4]
\arrow[from=2-1,to=2-2]
\arrow["\iota",from=2-2,to=2-3]
\arrow[from=2-3,to=2-4]
\arrow["\cong",from=1-1,to=2-1]
\arrow[from=1-2,to=2-2]
\arrow[from=1-3,to=2-3]
\arrow["\cong",from=1-4,to=2-4]
\arrow["\cong",shift left=3,from=2-1,to=1-1]
\arrow[shift left=3,from=2-2,to=1-2]
\arrow[shift left=3,from=2-3,to=1-3]
\arrow["\cong",shift left=3,from=2-4,to=1-4]
\end{tikzcd}
\end{equation}
The downward maps are induced by the graph, and the upward maps by its
transpose. On $H^1(\mathcal O)$ and $H^2(\mathcal O)$ they are mutually
inverse. Thus
Lemma~\ref{lemma: identification of kernel and cokernels}, applied to
this diagram, gives
\begin{equation}\label{eq: nygaard inclusion isos}
\iota:K_{\mathcal N}\xrightarrow{\simeq}K_{\mathrm{cris}},
\qquad
\iota:\mathcal C_{\mathcal N}\xrightarrow{\simeq}\mathcal C_{\mathrm{cris}},
\end{equation}
This proves the assertions for $\iota$. 
\end{proof}
To understand the maps induced by $\varphi_{(1)}$, we need to take into account also the second piece of the Nygaard filtration. We start with the following lemma.
\begin{lemma}\label{lemma: divided Frobenius low degrees}
For $r=1,2$, put
\[
Q_{r,T}:=\operatorname{coker}\bigl(
\varphi_{(r)}:\mathcal N^{\geq r}W\Omega_T^\bullet
\longrightarrow W\Omega_T^\bullet\bigr).
\]
Then $\mathcal H^j(Q_{r,T})=0$ for $j\leq r$, and there is a natural injection
\[
\mathbb H^2(T,Q_{1,T})
\hookrightarrow H^0\bigl(T,\mathcal H^2(\Omega_T^\bullet)\bigr).
\]
\end{lemma}

\begin{proof}
The map $\varphi_{(r)}$ is termwise injective, and its image is described
in \cite[Definition~8.1 and the proof of Lemma~8.2]{BMS2}. Hence
\[
Q_{r,T}^j=
\begin{cases}
0,&j<r,\\
W\Omega_T^j/p^{j-r}FW\Omega_T^j,&j\geq r.
\end{cases}
\]
A cycle in degree $r$ is locally represented by $x\in W\Omega_T^r$ with
$dx\in pFW\Omega_T^{r+1}$. By \cite[I, (3.21.1.5)]{Illusie},
\[
dx\in pFW\Omega_T^{r+1}\subseteq pW\Omega_T^{r+1}
\quad\Longrightarrow\quad x\in FW\Omega_T^r
\quad\Longrightarrow\quad [x]=0.
\]
This proves the vanishing in degrees $j\leq r$.

To compare $Q_{1,T}$ with the de Rham complex, use
$\varphi_{(1)}|_{\mathcal N^{\geq2}}=p\varphi_{(2)}$:
\[
\begin{tikzcd}[column sep=small]
0 &
{\mathcal N^{\geq2}W\Omega_T^\bullet} &
{W\Omega_T^\bullet} &
{Q_{2,T}} & 0 \\
0 &
{\mathcal N_T} &
{W\Omega_T^\bullet} &
{Q_{1,T}} & 0
\arrow[from=1-1,to=1-2]
\arrow["\varphi_{(2)}",from=1-2,to=1-3]
\arrow[from=1-3,to=1-4]
\arrow[from=1-4,to=1-5]
\arrow[from=2-1,to=2-2]
\arrow["\varphi_{(1)}",from=2-2,to=2-3]
\arrow[from=2-3,to=2-4]
\arrow[from=2-4,to=2-5]
\arrow[hookrightarrow,from=1-2,to=2-2]
\arrow["p",from=1-3,to=2-3]
\arrow["p",from=1-4,to=2-4]
\end{tikzcd}
\]
The right vertical map is injective. Indeed, its source vanishes
below degree two, and in degree $j\geq2$ the $p$-torsion-freeness of
$W\Omega_T^j$ gives
\[
px\in p^{j-1}FW\Omega_T^j
\quad\Longrightarrow\quad x\in p^{j-2}FW\Omega_T^j.
\]
Taking cofibers of the vertical maps gives a triangle
\[
\mathcal N_T/\mathcal N^{\geq2}W\Omega_T^\bullet
\xrightarrow{\varphi_{(1)}\bmod p}W\Omega_T^\bullet/p
\longrightarrow Q_{1,T}/pQ_{2,T}\longrightarrow.
\]
By \cite[Lemma~8.2]{BMS2}, the first map identifies, in the derived
category, with the canonical map
$\tau^{\leq1}\Omega_T^\bullet\to\Omega_T^\bullet$.
Thus
\[
Q_{1,T}/pQ_{2,T}\simeq\tau^{\geq2}\Omega_T^\bullet,
\qquad
Q_{2,T}\xrightarrow{p}Q_{1,T}
\longrightarrow\tau^{\geq2}\Omega_T^\bullet
\longrightarrow Q_{2,T}[1],
\]
where $\tau$ denotes canonical truncation. Since
$\mathbb H^2(T,Q_{2,T})=0$, the associated long exact sequence gives
\[
\mathbb H^2(T,Q_{1,T})
\hookrightarrow\mathbb H^2(T,\tau^{\geq2}\Omega_T^\bullet)
=H^0\bigl(T,\mathcal H^2(\Omega_T^\bullet)\bigr).
\]
The construction uses the compatible divided Frobenius and reduction
maps, so the injection has the asserted naturality.
\end{proof}

\begin{proof}[End of proof of Proposition~\ref{prop: nygaard comparison maps}]
The vanishing $\mathbb H^1(T,Q_{1,T})=0$ gives
\begin{equation}\label{eq: nygaard phi degree two}
0\longrightarrow\mathbb H^2(T,\mathcal N_T)
\xrightarrow{\varphi_{(1)}}H^2_{\mathrm{cris}}(T/W)
\longrightarrow\mathbb H^2(T,Q_{1,T}).
\end{equation}
By Cartier, the target of the injection in the previous lemma identifies
semilinearly with $H^0(T,\Omega_T^2)$.
The graph and transpose are mutually inverse on global two-forms.
Naturality of the injection therefore gives
\[
f^*f_*=\mathrm{id}\text{ on }\mathbb H^2(X,Q_{1,X}),
\qquad
f_*f^*=\mathrm{id}\text{ on }\mathbb H^2(Y,Q_{1,Y}).
\]
Both maps preserve the images of the last arrows in
\eqref{eq: nygaard phi degree two}; their restrictions to these images
are therefore mutually inverse as well.
Replacing the last terms by these images gives short exact rows with
an isomorphism on the right. The snake lemma yields
\[
\varphi_{(1)}:K_{\mathcal N}\xrightarrow{\simeq}K_{\mathrm{cris}},
\qquad
\varphi_{(1)}:\mathcal C_{\mathcal N}\xrightarrow{\simeq}\mathcal C_{\mathrm{cris}}.
\]
Together with \eqref{eq: nygaard inclusion isos}, we have proved
\begin{equation}\label{eq: nygaard two isos}
\begin{aligned}
\iota,\varphi_{(1)} &:K_{\mathcal N}\xrightarrow{\simeq}K_{\mathrm{cris}},\\
\iota,\varphi_{(1)} &:\mathcal C_{\mathcal N}\xrightarrow{\simeq}\mathcal C_{\mathrm{cris}}.
\end{aligned}
\end{equation}

\end{proof}

\subsection{Identification of the syntomic cokernel}
The maps induced by $\iota$ are $W$-linear, and those induced by
$\varphi_{(1)}$ are $\sigma$-semilinear. The four modules are finitely
generated over $W$ by properness. On either pair in \eqref{eq: nygaard two isos},
\[
\iota-\varphi_{(1)}=
(\mathrm{id}-\varphi_{(1)}\circ\iota^{-1})\circ\iota,
\]
and $\varphi_{(1)}\circ\iota^{-1}$ is a $\sigma$-semilinear
automorphism. Thus, Lemma~\ref{lemma: semilinear algebra over W torsion plus torsion free}
gives
\begin{equation}\label{eq: nygaard kernel surj}
\iota-\varphi_{(1)}\colon K_{\mathcal N}\longrightarrow K_{\mathrm{cris}}
\quad\text{surjective},
\end{equation}
and 
\begin{equation}\label{eq: nygaard fixed descent}
\ker\bigl(\iota-\varphi_{(1)}:
\mathcal C_{\mathcal N}\to\mathcal C_{\mathrm{cris}}\bigr)
\otimes_{\mathbf Z_p}W\xrightarrow{\simeq}\mathcal C_{\mathrm{cris}},
\qquad m\otimes w\longmapsto w\iota(m);
\end{equation}
we now identify this last kernel with $\mathcal C_{\mathrm{syn}}$.

By \cite[Proposition~8.4]{BMS2}, there is an exact sequence of complexes
on the pro-\'etale site
\begin{equation}\label{eq: nygaard syntomic complex}
0\longrightarrow W\Omega^1_{T,\log}[-1]
\longrightarrow\mathcal N_T
\xrightarrow{\iota-\varphi_{(1)}}W\Omega_T^\bullet
\longrightarrow0.
\end{equation}
Using $\mathbf Z_p(1)\simeq W\Omega^1_{T,\log}[-1]$, we obtain the
following diagram with exact rows:
\begin{equation}\label{diagram: nygaard syntomic}
\begin{tikzcd}[
column sep=small,
cells={font=\small},
labels={font=\scriptsize}
]
{H^1_{\mathrm{cris}}(X/W)} &
{H^2(X,\mathbf Z_p(1))} &
{\mathbb H^2(X,\mathcal N_X)} &
{H^2_{\mathrm{cris}}(X/W)} &
{H^3(X,\mathbf Z_p(1))_{\mathrm{tors}}} \\
{H^1_{\mathrm{cris}}(Y/W)} &
{H^2(Y,\mathbf Z_p(1))} &
{\mathbb H^2(Y,\mathcal N_Y)} &
{H^2_{\mathrm{cris}}(Y/W)} &
{H^3(Y,\mathbf Z_p(1))_{\mathrm{tors}}}
\arrow["\partial_X",from=1-1,to=1-2]
\arrow["j_X",from=1-2,to=1-3]
\arrow["\iota-\varphi_{(1)}",from=1-3,to=1-4]
\arrow[from=1-4,to=1-5]
\arrow["\partial_Y",from=2-1,to=2-2]
\arrow["j_Y",from=2-2,to=2-3]
\arrow["\iota-\varphi_{(1)}",from=2-3,to=2-4]
\arrow[from=2-4,to=2-5]
\arrow["\cong",from=1-1,to=2-1]
\arrow[from=1-2,to=2-2]
\arrow[from=1-3,to=2-3]
\arrow[from=1-4,to=2-4]
\arrow["\cong",from=1-5,to=2-5]
\end{tikzcd}
\end{equation}
Note that the boundary $H^2_{\mathrm{cris}}(T/W) \rightarrow H^3(T,\mathbf Z_p(1))$ takes values in torsion. Indeed, after
inverting $p$, the inclusion is an isomorphism and the middle map
becomes $1-\varphi/p$ on $H^2_{\mathrm{cris}}(T/W)[1/p]$.
This map is surjective by slope decomposition over the algebraically
closed field $k$.

The leftmost vertical map is an isomorphism by birational invariance
of $H^1_{\mathrm{cris}}$. For the rightmost one, we use the functorial
Brauer comparison
\[
H^3(T,\mathbf Z_p(1))_{\mathrm{tors}}
\simeq\operatorname{Br}(T)[p^\infty]/
\operatorname{Br}(T)[p^\infty]_{\mathrm{div}};
\]
see \cite[Appendix~A, Theorem~A.1 and its proof]{SkorobogatovPetrov}.
Birational invariance of the Brauer group therefore gives the indicated
isomorphism on the right.

\begin{prop}\label{prop: nygaard cokernel identification}
The maps $j_X,j_Y$ induce an isomorphism
\[
\mathcal C_{\mathrm{syn}}\xrightarrow{\simeq}
\ker\bigl(\iota-\varphi_{(1)}:
\mathcal C_{\mathcal N}\to\mathcal C_{\mathrm{cris}}\bigr).
\]
Consequently, the integral logarithmic class map induces an isomorphism
\begin{equation}\label{eq: integral crystalline cokernel syntomic}
\mathcal C_{\mathrm{syn}}\otimes_{\mathbf Z_p}W
\xrightarrow{\simeq}\mathcal C_{\mathrm{cris}}.
\end{equation}
\end{prop}

\begin{proof}
In \eqref{diagram: nygaard syntomic}, the first vertical map is
surjective. Hence $\ker(j_X)\to\ker(j_Y)$ is surjective, and the
snake lemma applied to
\[
0\longrightarrow\ker(j_T)\longrightarrow H^2(T,\mathbf Z_p(1))
\longrightarrow\operatorname{im}(j_T)\longrightarrow0
\]
identifies $\mathcal C_{\mathrm{syn}}$ with the cokernel of
$\operatorname{im}(j_X)\to\operatorname{im}(j_Y)$.

Apply Lemma~\ref{lemma: snake kernel cokernel} to the exact rows
\[
0\longrightarrow\operatorname{im}(j_T)
\longrightarrow\mathbb H^2(T,\mathcal N_T)
\xrightarrow{\iota-\varphi_{(1)}}H^2_{\mathrm{cris}}(T/W)
\longrightarrow H^3(T,\mathbf Z_p(1))_{\mathrm{tors}}.
\]
The map on the middle cokernels is surjective by
\eqref{eq: nygaard kernel surj}, and the map on the last terms is
injective by the Brauer comparison. The equality assertion of the
lemma therefore gives the first isomorphism. Equation
\eqref{eq: nygaard fixed descent} gives the second; its map is induced
by $\iota\circ j_T$, so it is the integral logarithmic class map.
\end{proof}

\section{Proof of the main theorem}\label{sec: final proof}
We have now all the ingredients to prove Theorem~\ref{main thm}. Since cokernels commute with reduction modulo $p$,
\eqref{eq: integral crystalline cokernel syntomic} gives
\begin{equation}\label{eq: coker alpha Ccris}
\operatorname{coker}(\alpha)\simeq \mathcal C_{\mathrm{cris}}/p
\simeq(\mathcal C_{\mathrm{syn}} /p)\otimes_{\mathbf F_p}k.
\end{equation}
We now identify $\mathcal C_{\mathrm{syn}} /p$ with the logarithmic cokernel.
The coefficient triangle for multiplication by $p$ gives the diagram
\[
\begin{tikzcd}[column sep=small,cells={font=\small}]
0 &
{H^2(X,\mathbf Z_p(1))/p} &
{H^2_{\mathrm{fppf}}(X,\mu_p)} &
{H^3(X,\mathbf Z_p(1))[p]} & 0 \\
0 &
{H^2(Y,\mathbf Z_p(1))/p} &
{H^2_{\mathrm{fppf}}(Y,\mu_p)} &
{H^3(Y,\mathbf Z_p(1))[p]} & 0
\arrow[from=1-1,to=1-2]
\arrow[from=1-2,to=1-3]
\arrow[from=1-3,to=1-4]
\arrow[from=1-4,to=1-5]
\arrow[from=2-1,to=2-2]
\arrow[from=2-2,to=2-3]
\arrow[from=2-3,to=2-4]
\arrow[from=2-4,to=2-5]
\arrow[from=1-2,to=2-2]
\arrow[from=1-3,to=2-3]
\arrow["\cong",from=1-4,to=2-4]
\end{tikzcd}
\]
with exact rows; see
\cite[Appendix~A, (A.3)--(A.4)]{SkorobogatovPetrov}.
The rightmost vertical map is an isomorphism by the same Brauer
comparison. The snake lemma therefore identifies the cokernels of the
first two vertical maps:
\begin{equation}\label{eq: syntomic mod p mu p}
\mathcal C_{\mathrm{syn}} /p\simeq\operatorname{coker}\bigl(
H^2_{\mathrm{fppf}}(X,\mu_p)\to
H^2_{\mathrm{fppf}}(Y,\mu_p)\bigr).
\end{equation}
Using $H^2_{\mathrm{fppf}}(T,\mu_p)\simeq
H^1(T,\Omega^1_{T,\log})$ (see \cite{Gros}), we obtain
\begin{equation}\label{eq: coker alpha logarithmic}
\operatorname{coker}(\alpha)\simeq
\operatorname{coker}\bigl(
H^1(X,\Omega^1_{X,\log})\to
H^1(Y,\Omega^1_{Y,\log})\bigr)\otimes_{\mathbf F_p}k.
\end{equation}

\begin{lemma}\label{lem: cokernel comparison}
The natural map $\operatorname{coker}(\alpha)\to\operatorname{coker}(\beta)$ in
\eqref{diagram: universal coefficient} is injective.
\end{lemma}

\begin{proof}
We must check that \eqref{eq: coker alpha logarithmic} identifies the
\emph{natural} map in the statement with the injective
map appearing in Proposition \ref{prop: logarithmic description de Rham}. The integral comparison is induced by
\[
\mathbf Z_p(1)\simeq W\Omega^1_{T,\log}[-1]
\longrightarrow\mathcal N_T
\xrightarrow{\iota}W\Omega_T^\bullet.
\]
Its reduction modulo $p$ is the class map
$\Omega^1_{T,\log}[-1]\to\Omega_T^\bullet$.
Consequently we have a commutative diagram
\[
\begin{tikzcd}[column sep=large]
{\operatorname{coker}\bigl(
H^1(X,\Omega^1_{X,\log})\to
H^1(Y,\Omega^1_{Y,\log})\bigr)\otimes_{\mathbf F_p}k}
& {\operatorname{coker}(\alpha)} \\
& {\operatorname{coker}(\beta)}
\arrow["\simeq",from=1-1,to=1-2]
\arrow[hookrightarrow,from=1-1,to=2-2]
\arrow[from=1-2,to=2-2]
\end{tikzcd}
\]
whose diagonal map is injective by
Proposition~\ref{prop: logarithmic description de Rham}.
The right-hand map is therefore injective.
\end{proof}
\begin{proof}[Proof of Theorem~\ref{main thm}]

 By
Proposition~\ref{prop: kernel comparison}, the map
$\ker(\alpha)\to\ker(\beta)$ is surjective. By
Lemma~\ref{lem: cokernel comparison}, the map
$\operatorname{coker}(\alpha)\to\operatorname{coker}(\beta)$ is injective. The snake lemma therefore gives
\[
\ker\bigl(H^3_{\mathrm{cris}}(X/W)[p]\to
H^3_{\mathrm{cris}}(Y/W)[p]\bigr)=0.
\]
Thus the graph correspondence is injective on the $p$-torsion. Applying the
same argument to the transpose correspondence gives an injection in the
opposite direction. Set
\[
T_X:=\operatorname{Tors}H^3_{\mathrm{cris}}(X/W),
\qquad
T_Y:=\operatorname{Tors}H^3_{\mathrm{cris}}(Y/W).
\]
These are finite-length $W$-modules. We claim that the graph correspondence
$T_X\to T_Y$ is injective. Suppose that $0\neq x$ belongs to its kernel and
choose $r\geq1$ minimal with $p^rx=0$. Then
$0\neq p^{r-1}x\in T_X[p]$ is also in the kernel, contradicting the
injectivity on $p$-torsion. Hence $T_X\to T_Y$ is injective. The transpose
correspondence similarly gives an injection $T_Y\to T_X$.

Since both modules have finite length,
\[
\operatorname{length}_W(T_X)\leq\operatorname{length}_W(T_Y)
\leq\operatorname{length}_W(T_X).
\]
The lengths are therefore equal, and both injections are isomorphisms. This
proves
\[
\operatorname{Tors}H^3_{\mathrm{cris}}(X/W)
\simeq
\operatorname{Tors}H^3_{\mathrm{cris}}(Y/W).
\]
\end{proof}

\appendix
\section{Commutative algebra proofs}

In this appendix, we collect some elementary arguments repeatedly used in the text.

\begin{lemma}\label{lemma: identification of kernel and cokernels}
Consider the following diagram of abelian groups with exact rows,
where $a_1$ and $d_1$ are isomorphisms and all squares commute
in both directions:
\[
\begin{tikzcd}
A_1 \arrow[r,"\ell_1"]
    \arrow[d,"a_1",shift left=2]
& B_1 \arrow[r,"m_1"]
      \arrow[d,"b_1",shift left=2]
& C_1 \arrow[r,"n_1"]
      \arrow[d,"c_1",shift left=2]
& D_1 \arrow[d,"d_1",shift left=2] \\
A_2 \arrow[r,"\ell_2"]
    \arrow[u,"a_1^{-1}",shift left=2]
& B_2 \arrow[r,"m_2"]
      \arrow[u,"b_2",shift left=2]
& C_2 \arrow[r,"n_2"]
      \arrow[u,"c_2",shift left=2]
& D_2 \arrow[u,"d_1^{-1}",shift left=2].
\end{tikzcd}
\]
Then the horizontal maps induce isomorphisms
\[
\ker(b_1)\simeq\ker(c_1),
\qquad
\operatorname{coker}(b_1)\simeq\operatorname{coker}(c_1).
\]
\end{lemma}

\begin{proof}
Set $I_i=\operatorname{im}(\ell_i)$,
$J_i=\operatorname{im}(m_i)$, and
$Q_i=\operatorname{im}(n_i)$. The squares involving $a_1$ and $a_1^{-1}$ show that
$b_1:I_1\to I_2$ and $b_2:I_2\to I_1$ are inverse
isomorphisms. Likewise, the squares involving $d_1$ and
$d_1^{-1}$ show that $d_1:Q_1\to Q_2$ and
$d_1^{-1}:Q_2\to Q_1$ are inverse isomorphisms. Apply the snake lemma to the short exact sequences
$0\to I_i\to B_i\to J_i\to0$ and
$0\to J_i\to C_i\to Q_i\to0$ and conclude. 
\end{proof}
The following form of the snake lemma is used in both the Cartier
and Nygaard arguments.
\begin{lemma}\label{lemma: snake kernel cokernel}
Consider a commutative diagram of abelian groups with exact rows
\[
\begin{tikzcd}
0 \arrow[r] & A_1 \arrow[r] \arrow[d,"u"]
& B_1 \arrow[r,"d_1"] \arrow[d,"v"]
& D_1 \arrow[r] \arrow[d,"w"]
& E_1 \arrow[d,"z"] \\
0 \arrow[r] & A_2 \arrow[r]
& B_2 \arrow[r,"d_2"]
& D_2 \arrow[r]
& E_2 .
\end{tikzcd}
\]
Then $\ker(u)\simeq\ker(\ker(v)\to\ker(w))$.
If $\ker(v)\to\ker(w)$ is surjective, the natural map induces an injection
\[
\operatorname{coker}(u)\hookrightarrow
\ker\bigl(\operatorname{coker}(v)\to\operatorname{coker}(w)\bigr).
\]
If, in addition, $z$ is injective, this injection is an isomorphism.
\end{lemma}

\begin{proof}
Set $I_i=\operatorname{im}(d_i)$ and let $t:I_1\to I_2$ be the
induced map. Applying the snake lemma to
$0\to A_i\to B_i\to I_i\to0$ gives the kernel assertion and
\[
\ker(v)\longrightarrow\ker(t)\longrightarrow\operatorname{coker}(u)
\longrightarrow\operatorname{coker}(v)
\longrightarrow\operatorname{coker}(t)\longrightarrow0.
\]
Since $\ker(t)\subseteq\ker(w)$, the surjectivity hypothesis makes
the first arrow surjective. Thus $\operatorname{coker}(u)$ embeds in
$\operatorname{coker}(v)$ with image
$\ker(\operatorname{coker}(v)\to\operatorname{coker}(t))$.
This image is contained in the kernel stated in the lemma. If $z$ is injective, the induced map $D_1/I_1\to D_2/I_2$ is
injective, since $D_i/I_i$ embeds in $E_i$. A second application of
the snake lemma, to $0\to I_i\to D_i\to D_i/I_i\to0$, shows that
$\operatorname{coker}(t)\to\operatorname{coker}(w)$ is injective.
The two kernels therefore agree.
\end{proof}

We now turn to semilinear algebra. Throughout the remainder
of this appendix, let $k$ be an algebraically closed field
of characteristic $p>0$, write $W=W(k)$ and $W_n=W_n(k)$,
and let $\sigma$ denote the Witt-vector Frobenius.
All semilinear maps are taken with respect to $\sigma$.
We write $M^{\varphi=1}=\ker(1-\varphi)$.

\begin{lemma}\label{lemma: semilinear algebra finite length}
Let $n\geq1$, and let $M$ be a finite-length $W_n$-module
equipped with a semilinear map $\varphi:M\to M$.
Then the natural map
\[
M^{\varphi=1}\otimes_{\mathbf Z/p^n\mathbf Z}W_n
\longrightarrow M,\qquad m\otimes a\longmapsto am,
\]
is injective. If $\varphi$ is bijective, this map is an
isomorphism, and $1-\varphi:M\to M$ is surjective.
\end{lemma}

\begin{proof}
We first recall the case $n=1$. Fixed vectors that are
linearly independent over $\mathbf F_p$ are linearly
independent over $k$. Indeed, in a dependence relation
of minimal length, normalize one coefficient to $1$.
Applying $\varphi$ and subtracting the original relation
shows that all coefficients lie in $\mathbf F_p$,
a contradiction.

If $\varphi$ is bijective, then $1-\varphi$ is surjective
and its kernel has $\mathbf F_p$-dimension $\dim_k M$;
see \cite[Tag~0A3L]{Stacks}.
Together with the preceding independence argument,
this proves the lemma for $n=1$.

For general $n$, let $N=M^{\varphi=1}$ and denote the
natural map by $\lambda$. Since $W$ is torsion-free,
it is flat over $\mathbf Z_p$, and hence $W_n$ is flat
over $\mathbf Z/p^n\mathbf Z$.
Writing $M[p]=\ker(p:M\to M)$, flatness identifies the
$p$-torsion of $N\otimes_{\mathbf Z/p^n\mathbf Z}W_n$
with $(M[p])^{\varphi=1}\otimes_{\mathbf F_p}k$.
The restriction of $\lambda$ to this submodule is
injective by the case $n=1$. Every nonzero module
annihilated by $p^n$ contains nonzero $p$-torsion,
so $\lambda$ itself is injective.

Suppose now that $\varphi$ is bijective. It induces a
bijective semilinear map on each $p^jM/p^{j+1}M$.
The vector-space case and induction along this finite
filtration show that $1-\varphi$ is surjective on both
$M$ and $pM$. The latter implies that every fixed
vector in $M/pM$ lifts to a fixed vector in $M$.
Choose such lifts of a fixed basis of $M/pM$.
They generate $M$ by Nakayama's lemma, proving that
$\lambda$ is surjective.
\end{proof}

\begin{lemma}\label{lemma: semilinear algebra over W}
Let $M$ be a finite free $W$-module equipped with a
semilinear isomorphism $\varphi:M\to M$.
Then there is a short exact sequence
\[
0\longrightarrow M^{\varphi=1}
\longrightarrow M
\xrightarrow{1-\varphi}M
\longrightarrow0,
\]
and the natural map
$M^{\varphi=1}\otimes_{\mathbf Z_p}W\to M$
is an isomorphism.
\end{lemma}

\begin{proof}
By Lemma \ref{lemma: semilinear algebra finite length},
$1-\varphi$ is surjective on each
$p^jM/p^{j+1}M$. Successive approximation and
$p$-adic completeness therefore show that it is
surjective on $M$ and on $pM$: a solution modulo $p^j$
can be corrected by an element of $p^jM$ to obtain
a solution modulo $p^{j+1}$.
This proves exactness and shows that every fixed
vector modulo $p$ lifts to a fixed vector in $M$.

Choose fixed lifts $e_1,\ldots,e_r$ of a fixed
$k$-basis of $M/pM$. They form a $W$-basis of $M$
by Nakayama's lemma. Since $W^{\sigma=1}=\mathbf Z_p$,
we have $M^{\varphi=1}=\bigoplus_{i=1}^r\mathbf Z_p e_i$.
The asserted isomorphism follows.
\end{proof}

The same conclusions hold for finitely generated
$W$-modules. To prove this, we use the torsion
submodule, which is preserved by $\varphi$.

\begin{lemma}\label{lemma: semilinear algebra over W torsion plus torsion free}
Let $M$ be a finitely generated $W$-module equipped
with a semilinear isomorphism $\varphi:M\to M$.
Then there is a short exact sequence
\[
0\longrightarrow M^{\varphi=1}
\longrightarrow M
\xrightarrow{1-\varphi}M
\longrightarrow0,
\]
and the natural map
$M^{\varphi=1}\otimes_{\mathbf Z_p}W\to M$
is an isomorphism.
\end{lemma}

\begin{proof}
Let $T$ be the torsion submodule of $M$, and set
$Q=M/T$. Then $T$ is a finite-length $W_n$-module
for some $n\geq1$, and $Q$ is finite free over $W$.
Both inherit semilinear isomorphisms from $\varphi$.

By the preceding two lemmas, $1-\varphi$ is
surjective on $T$ and $Q$, and hence on $M$.
The snake lemma also gives the exact sequence
$0\to T^{\varphi=1}\to M^{\varphi=1}
\to Q^{\varphi=1}\to0$.

Tensor this sequence with $W$ over $\mathbf Z_p$,
which preserves exactness, and compare it with
$0\to T\to M\to Q\to0$.
The natural maps on the two ends are isomorphisms
by the preceding lemmas. Therefore the middle map
is an isomorphism as well.
\end{proof}

\bibliographystyle{plain}
\bibliography{biblioteca}
\end{document}